\documentclass[12pt,reqno]{article}

\usepackage[usenames]{color}
\usepackage{amssymb}
\usepackage{amsmath}
\usepackage{amsthm}
\usepackage{amsfonts}
\usepackage{amscd}
\usepackage{graphicx}
\usepackage{lscape}
\usepackage[title]{appendix}

\usepackage[colorlinks=true,
linkcolor=webgreen,
filecolor=webbrown,
citecolor=webgreen]{hyperref}

\definecolor{webgreen}{rgb}{0,.5,0}
\definecolor{webbrown}{rgb}{.6,0,0}

\usepackage{color}
\usepackage{fullpage}
\usepackage{float}

\usepackage{graphics}
\usepackage{latexsym}
\usepackage{epsf}
\usepackage{breakurl}

\newcommand{\seqnum}[1]{\href{https://oeis.org/#1}{\rm \underline{#1}}}
\def\suchthat{\, : \, }

\def\modd#1 #2{#1\ \mbox{\rm (mod}\ #2\mbox{\rm )}}
\def\uni{ \, \cup \, }

\newenvironment{smallarray}[1]
{\null\,\vcenter\bgroup\scriptsize
\arraycolsep=.13885em
\hbox\bgroup$\array{@{}#1@{}}}
{\endarray$\egroup\egroup\,\null}

\begin{document}

\theoremstyle{plain}
\newtheorem{theorem}{Theorem}
\newtheorem{corollary}[theorem]{Corollary}
\newtheorem{lemma}[theorem]{Lemma}
\newtheorem{proposition}[theorem]{Proposition}

\theoremstyle{definition}
\newtheorem{definition}[theorem]{Definition}
\newtheorem{example}[theorem]{Example}
\newtheorem{conjecture}[theorem]{Conjecture}
\newtheorem{openproblem}[theorem]{Open Problem}

\theoremstyle{remark}
\newtheorem{remark}[theorem]{Remark}

\title{Proof of a conjecture of Krawchuk and Rampersad}

\author{
Jeffrey Shallit\footnote{Research funded by a grant from NSERC, 2018-04118.} \\
School of Computer Science\\
University of Waterloo\\
Waterloo, ON  N2L 3G1\\
Canada\\
\href{mailto:shallit@uwaterloo.ca}{\tt shallit@uwaterloo.ca}\\
}

\maketitle

\begin{abstract}
We prove a 2018 conjecture of Krawchuk and Rampersad on the extremal
behavior of $c(n)$, where $c(n)$ counts the number of length-$n$ factors 
of the Thue-Morse word $\bf t$, up to cyclic rotation.
\end{abstract}

\section{Introduction}

Let $\bf x$ be an infinite word.
In a recent paper, Krawchuk and Rampersad \cite{Krawchuk&Rampersad:2018}
studied the cyclic complexity function $c_{\bf x}(n)$, defined to be the number
of length-$n$ factors of $\bf x$, where factors that are the same, up to
cyclic shift, are only counted once.

They observed that for ${\bf t} = 01101001 \cdots$, the
Thue-Morse sequence \cite{Allouche&Shallit:1999},
the function $c_{\bf t}(n)$ is $2$-regular and is specified
by a linear representation of rank $50$.  This means there are
vectors $v, w$ and a matrix-valued morphism $\gamma$ such
that $c_{\bf t}(n) = v \gamma(z) w$ for all strings $z$
that are binary representations of $n$ (allowing leading zeros).
See \cite{Berstel&Reutenauer:2011} for more details.
The first few terms of $c_{\bf t}(n)$ are given
in Table~\ref{tab1}; it is sequence \seqnum{A360104} in the 
{\it On-Line Encyclopedia of Integer Sequences} (OEIS) \cite{Sloane:2023}.

\begin{table}[H]
\begin{center}
\begin{tabular}{c|ccccccccccccccccccccc}
$n$ & 0& 1& 2& 3& 4& 5& 6& 7& 8& 9&10&11&12&13&14&15&16&17&18&19\\
\hline
$c_{\bf t} (n)$ &  1& 2& 3& 2& 4& 4& 6& 8&12& 8&12&16&14&18&18&18&28&20&20&28
\end{tabular}
\end{center}
\caption{First few values of $c_{\bf t}(n)$.}
\label{tab1}
\end{table}

Krawchuk and Rampersad conjectured that 
$\limsup c_{\bf t}(n)/n = 2$ and
$\liminf c_{\bf t}(n)/n = 4/3$.
In this paper we prove these two conjectures.
The conjectures follow from two inequalities:
$c_{\bf t}(n) \leq 2n-4$ for $n \geq 3$ and
$c_{\bf t}(n) \geq {4\over 3}n - 4$ for $n \geq 0$, which we prove
in Section~\ref{three} and \ref{four}, respectively.

The method of linear representations (as discussed in, for example,
\cite{Shallit:2022}) is extremely powerful for proving statements
about automatic sequences, but it has some limitations.   While
it can usually be used to prove various equalities, proving
inequalities is typically more problematic.   In this paper we use
traditional techniques based on induction, together with linear
representations, to prove the desired inequalities.

\section{Preliminary results}
\label{two}

Throughout we abbreviate $c_{\bf t} (n)$ by $c(n)$.

\begin{proposition}
We have
\begin{align}
c(2^k) &= 2^{k+1} - 4 , && \quad\quad (k \geq 2) \label{2k}\\
c(2^k+3) &= {5\over 3}\cdot 2^k - {2\over 3} (-1)^k + 2, && \quad\quad (k \geq 2)
\label{2k3}\\
c(2^k+1) &= {4 \over 3} \cdot 2^k + {2\over 3} (-1)^k - 2, && \quad\quad (k \geq 2)
\label{2k1}\\
c(2^k-3) &= {5 \over 3} \cdot 2^k + {1\over 3} (-1)^k - 5, && \quad\quad (k \geq 5)
\label{2km3}\\
c(2^k-1) &= {4 \over 3} \cdot 2^k - {1\over 3} (-1)^k - 3, && \quad\quad (k \geq 2)
\label{2km1} \\
c(2^k - 5) &= {3\over 2} \cdot 2^k + (-1)^k - 7, && \quad\quad (k \geq 5) 
\label{2km5}\\
c(2^k - 7) &= {3 \over 2} \cdot 2^k - (-1)^k - 9, && \quad\quad (k \geq 3) 
\label{2km7} \\
c(12\cdot 2^k - 3) &= {{56} \over 3} \cdot 2^k - {2 \over 3} (-1)^k -10, && \quad\quad (k \geq 0). \label{12km3}
\end{align}
\end{proposition}

\begin{proof}
Eq.~\eqref{2k} can be found in \cite[Prop.~1]{Krawchuk&Rampersad:2018}.

The remaining equalities can be proved exactly as in that paper, using
the same technique.
\end{proof}

The following three identities are crucial to our approach.
\begin{lemma}
There are $2$-automatic sequences $a_0, a_1, a_3$ such that
\begin{align}
c(2i) &= 2c(i) + a_0 (i)  \label{eq0}\\
c(4i+1) &= 2c(i+1) + c(2i+1) + a_1(i) \label{eq1} \\
c(4i+3) &= {1 \over 2} c(2i) + c(2i+3) + {1\over 2} c(2i+4) + a_3(i)
\label{eq3}  
\end{align}
and furthermore 
\begin{align}
2 & \leq a_0(i) \leq 6, &&\quad\quad (i\geq 3) \label{b0}\\
0 & \leq a_1(i) \leq 10, && \quad\quad (i\geq 2) \label{b1}\\
-1 & \leq a_3(i) \leq 3, &&\quad\quad (i\geq 1) \label{b3}.
\end{align}
\end{lemma}

\begin{proof}
For each relation, we compute a linear representation for each
term except $a_0$ (resp., $a_1, a_3$) using {\tt Walnut} 
\cite{Mousavi:2016}; then we compute a linear
representation for the difference using block matrices.  Then
we minimize the linear representation and use the ``semigroup trick''
(see, e.g., \cite[\S 4.11]{Shallit:2022}) to verify that
the sequence is automatic and find a
deterministic finite automaton with output (DFAO) for $a_0$ (resp., $a_1, a_3$).

We provide more details about the computation of $a_0$ (in part 
because we will need them in what follows).  We start with
the following
{\tt Walnut} code:
\begin{verbatim}
def tmfactoreq "At t<n => T[i+t]=T[j+t]":
def tmconj "Et t<=n & $tmfactoreq(j,i+t,n-t) & $tmfactoreq(i,(j+n)-t,t)":
def tmc "Aj j<i => ~$tmconj(i,j,n)":
\end{verbatim}
Here {\tt tmc(i,n)} asserts that ${\bf t}[i..i+n-1]$ is the
first occurrence of a factor that is cyclically equivalent to it.
Hence, by counting the number of $i$ for which it evaluates to
{\tt TRUE}, we determine the number of length-$n$ cyclic factors.
Let us compute the number of cyclic factors of various lengths
using {\tt Walnut}:
\begin{verbatim}
eval tmcc n "$tmc(i,n)":
eval tmcc2 n "$tmc(i,2*n)":
\end{verbatim}


The first two commands produce linear representations for $c(n)$ and
$c(2n)$, in {\tt Maple} format.  They are of rank $50$ and
$60$, respectively.  From this we can easily compute a linear
representation for $a_0(n) := c(2n) - 2c(n)$, of rank $110$.
Using Sch\"utzenberger's algorithm \cite[Chap.~2]{Berstel&Reutenauer:2011},
this representation can be minimized into a linear representation
$(v, \gamma, w)$ of rank $7$, as follows:
\begin{align*}
v^T &= \left[ \begin{smallarray}{ccccccc}
1&0&0&0&0&0&0
	\end{smallarray} \right] 
&
\gamma(0) & = \left[ \begin{smallarray}{ccccccc}
1&0&0&0&0&0&0\\
 0&0&1&0&0&0&0\\
 0&0&0&0&1&0&0\\
 0&0&0&0&0&1&0\\
 0&0&0&0&1&0&0\\
 0&0&0&1&0&0&0\\
 0&0&0&0&0&0&1
	\end{smallarray} \right]
&
\gamma(1) &= \left[\begin{smallarray}{ccccccc} 
 0&  1&  0&  0&  0&  0&  0\\
 0&  0&  0&  1&  0&  0&  0\\
 0&  0&  0&  0&  1&  0&  0\\
 0&  0&  0&  0&  1/2&  0&  0\\[1pt]
 0&  0&  0&  0&  1&  0&  0\\
 0&  0&  0&  0&  0&  0&  1\\[1pt]
 0&  0&  0&  0&  1/2&  0&  0
\end{smallarray}\right]
&
w &= \left[ \begin{smallarray}{c}
-1\\
-1\\
-2\\
2\\
4\\
2\\
6
\end{smallarray} 
\right] .
\end{align*}
We can now use the ``semigroup trick'' to show that $a_0(n)$ is
$2$-automatic and find a
DFAO for it.  It has $8$ states and
is displayed in Figure~\ref{aut0}.
From examining the result, we see that $a_0(n) \in \{-2, -1, 2, 4, 6\}$
and furthermore $a_0(n) \in \{ 2, 4, 6\}$ for $n \geq 3$.
This proves Eqs.~\eqref{eq0} and \eqref{b0}.

\begin{figure}[H]
\begin{center}
\includegraphics[width=6in]{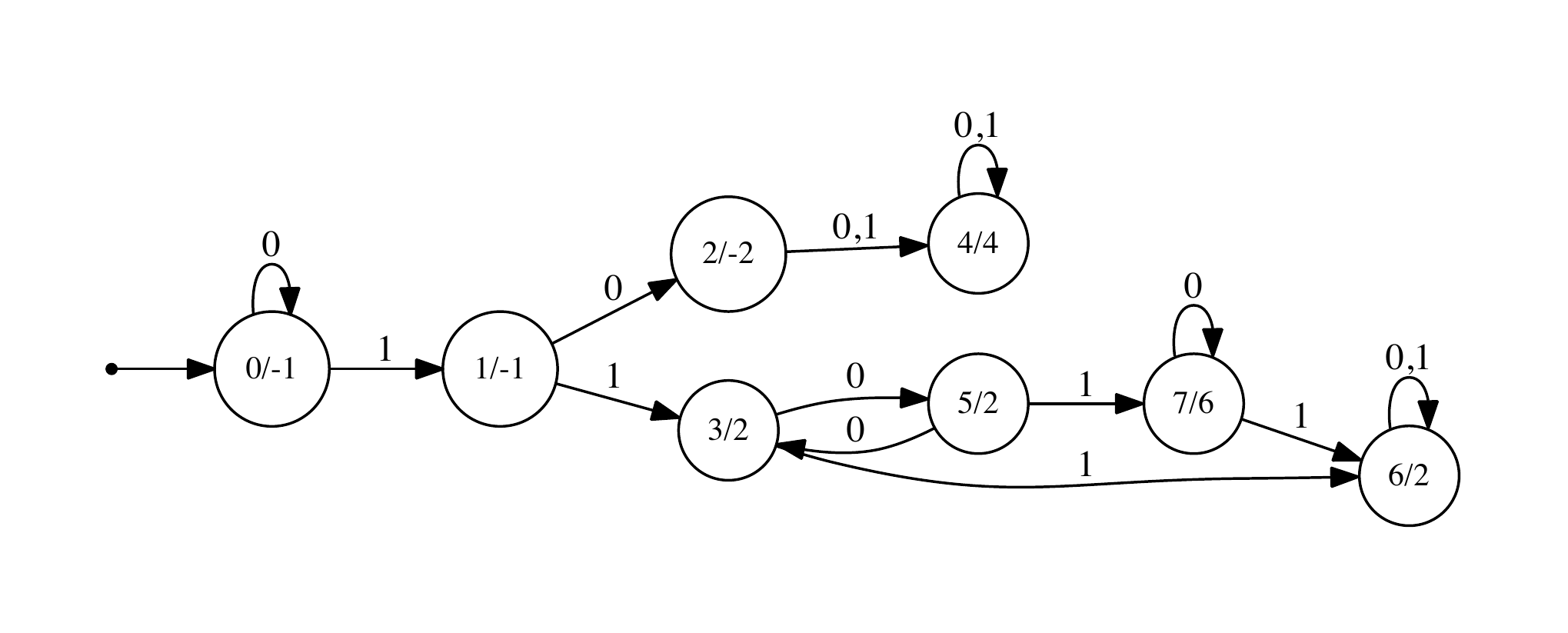}
\end{center}
\caption{DFAO for $a_0 (n)$.}
\label{aut0}
\end{figure}

Using a similar procedure, and the following {\tt Walnut} commands,
we can find the
DFAO's for $a_1 (n)$ and $a_3(n)$, which are given in Figures~\ref{aut1}
and \ref{aut3}.
\begin{verbatim}
eval tmcc41 n "$tmc(i,4*n+1)":
eval tmcc1 n "$tmc(i,n+1)":
eval tmcc21 n "$tmc(i,2*n+1)":
eval tmcc43 n "$tmc(i,4*n+3)":
eval tmcc23 n "$tmc(i,2*n+3)":
eval tmcc24 n "$tmc(i,2*n+4)":
\end{verbatim}

\begin{figure}[H]
\begin{center}
\includegraphics[width=6.5in]{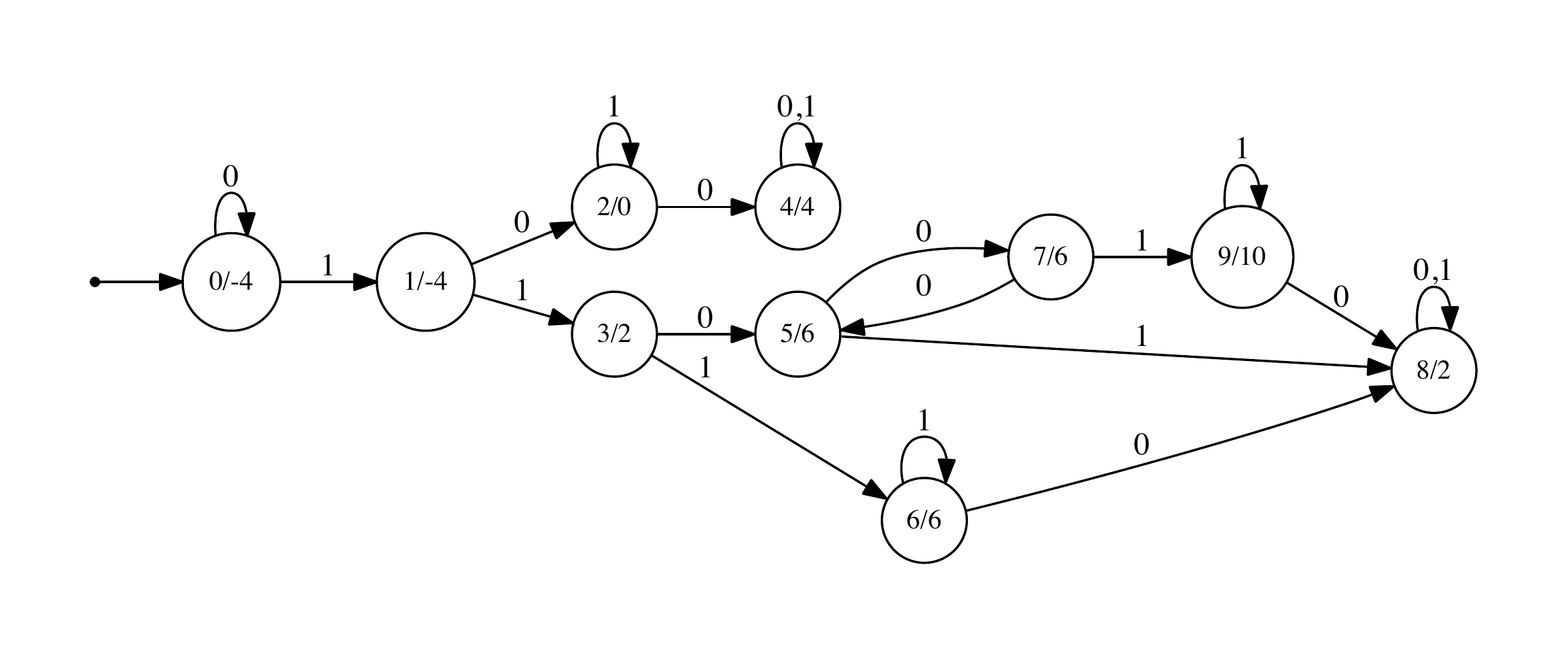}
\end{center}
\caption{DFAO for $a_1 (n)$.}
\label{aut1}
\end{figure}

\begin{figure}[H]
\begin{center}
\includegraphics[width=6.5in]{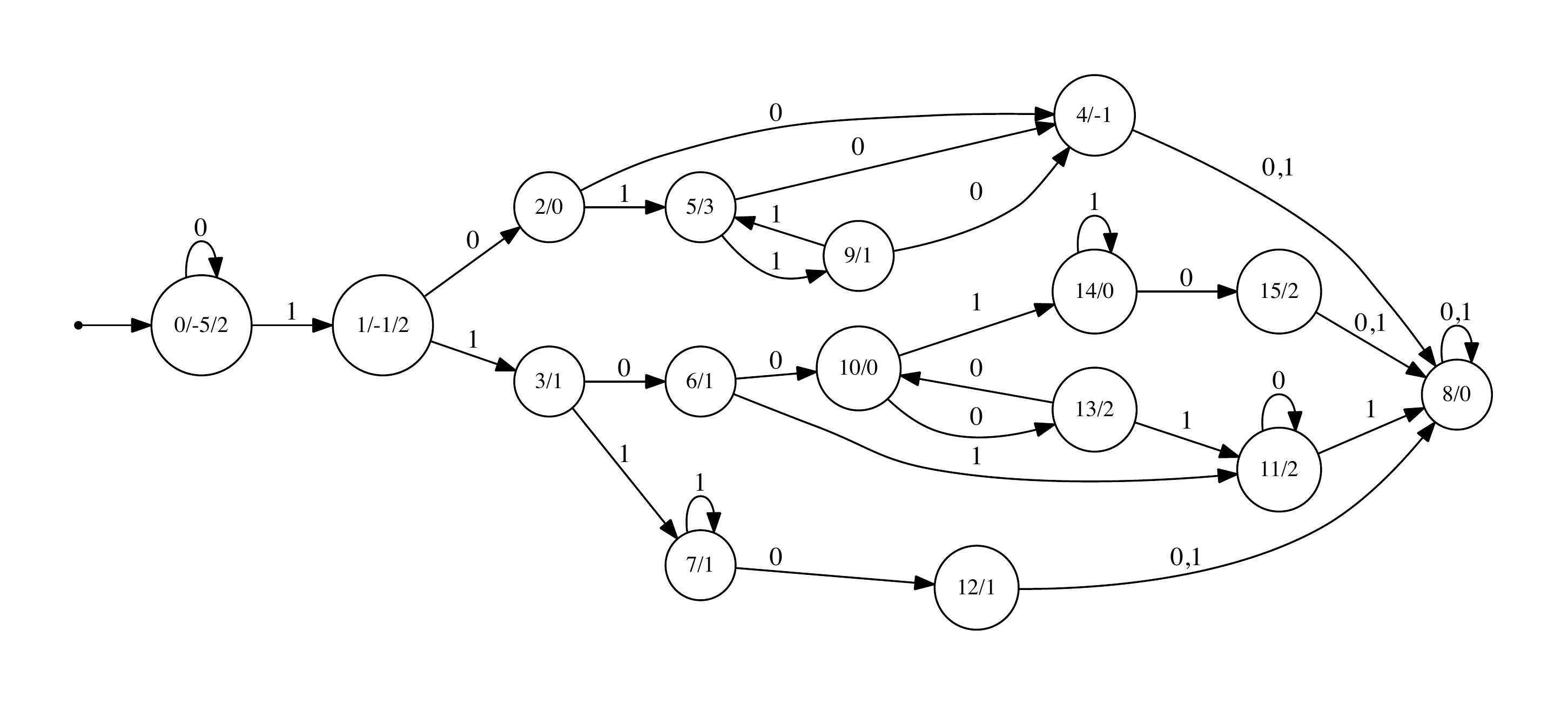}
\end{center}
\caption{DFAO for $a_3 (n)$.}
\label{aut3}
\end{figure}
Examining the results proves Eqs.~\eqref{eq1}, \eqref{eq3}, \eqref{b1},
\eqref{b3}.
\end{proof}

\section{Upper bound}
\label{three}

\begin{theorem}
$c(n) \leq 2n-4$ for all $n \geq 3$.
\end{theorem}

\begin{proof}
By induction on $n$, using Eqs.~\eqref{eq0}--\eqref{eq3} and
\eqref{b1}--\eqref{b3}.   However, the claim of the theorem does not seem
to be strong enough to carry out the induction, so we actually prove
the following stronger claim by induction:
\begin{equation}
c(n) \leq 2n-7 \text{ if } n \geq 12 \text{ and } n \not\in P_2,
\label{induc1}
\end{equation}
where $P_2 = \{ 2^i \suchthat i \geq 0 \} = \{1,2,4,8,16,\ldots \}$.
The base case is $n \leq 44$; we can easily check that 
Assertion~\eqref{induc1}
holds for these $n$.

Now assume $n \geq 45$.  There are three
cases to consider.

\bigskip

\noindent {\it Case 1:}
$n \equiv \modd{0} {2}$, $n = 2i$.  

Case 1a:  Assume
$n/2 \not\in P_2$.   Then 
\begin{align*}
c(n) &= c(2i) = 2c(i) + a_0 (i)\\ 
&=  2 c(n/2) + a_0 (n/2) \\
&\leq  2c(n/2) + 6  \\
&\leq 2 (n-7) + 6 \quad\quad \text{(by induction, since $n/2 \geq 22$)} \\
&= 2n-8 \leq 2n-7.
\end{align*}

Case 1b:  If $n/2 \in P_2$, then $n= 2^k$ for some $k \geq 1$.  Since
$n \geq 45$ we have $k \geq 6$.  Hence the desired
bound $c(n) \leq 2n-4$ follows from Eq.~\eqref{2k}.

\bigskip

\noindent{\it Case 2:}  $n \equiv \modd{1} {4}$, $n = 4i+1$.

\smallskip

Case 2a:  Assume 
$(n+3)/4 \not\in P_2$ and $(n+1)/2 \not\in P_2$.
Then we have
\begin{align*}
c(n) &= c(4i+1) = 2c(i+1) + c(2i+1) + a_1(i) \\
&\leq 2c( (n+3)/4) + c( (n+1)/2) + a_1((n-1)/4) \\
&\leq 2 ( (n+3)/2 - 7) + (n+1 - 7) + 10 \quad\quad \text{(by induction,
since $(n+1)/2 \geq (n+3)/4 \geq 12$)} \\
&= 2n -7 .
\end{align*}

Case 2b: If $(n+3)/4 \in P_2$, then $n = 2^k - 3$ for some $k \geq 2$. 
Since $n \geq 45$ we must have $k \geq 6$.   
Then Eq.~\eqref{2km3} implies
the desired bound.

\smallskip

Case 2c:  If $(n+1)/2 \in P_2$ then $n = 2^k - 1$ for some $k \geq 1$.
Since $n \geq 45$ we have $k \geq 6$.   Hence
by Eq.~\eqref{2km1} we get
the desired bound.

\bigskip

\noindent{\it Case 3:}  $n \equiv \modd{3} {4}$, $n = 4i+3$.   

\smallskip

Case 3a:  Assume
$(n-3)/2 \not\in P_2$ and $(n+3)/2 \not\in P_2$ and $(n+5)/2 \not\in P_2$.
Then we have
\begin{align*}
c(n) &= c(4i+3) = {1 \over 2} c(2i) + c(2i+3) + {1\over 2} c(2i+4) + a_3(i) \\
&= {1\over 2} c((n-3)/2) + c((n+3)/2) + {1 \over 2} c((n+5)/2) + a_3 ((n-3)/4) \\
&\leq {1\over 2} (n-10) + (n+3-7) + {1\over2} (n+5-7) + 3  \\
&\quad\quad \text{(by induction, since $(n+5)/2\geq (n+3)/2 \geq (n-3)/2 \geq 21$)} \\
&= 2n -7.
\end{align*}

\smallskip

Case 3b:  If $(n-3)/2 \in P_2$ then $n = 2^k + 3$ for some $k\geq 1$.  Since
$n \geq 45$ we have $k \geq 6$.  So
the desired bound follows from
Eq.~\eqref{2k3}.

\smallskip

Case 3c:  If $(n+3)/2 \in P_2$ then $n = 2^k - 3$ for some $k \geq 1$.
Since $n \geq 45$ we have $k \geq 6$.
So the desired bound follows from
Eq.~\eqref{2km3}.

\smallskip

Case 3d:  If $(n+5)/2 \in P_2$ then $n = 2^k - 5$ for some $k \geq 1$.
Since $k \geq 45$ we have $k \geq 6$.
So the desired bound follows from
Eq.~\eqref{2km5}.

\bigskip

We have now completed the proof of Assertion~\eqref{induc1}.  To
finish the proof of the theorem, we only need observe that 
if $n \geq 8$ is a power of $2$, then $c(n) = 2n-4$,
and check that $c(n) \leq 2n-4$ for $3 \leq n \leq 11$.
\end{proof}

We now get the first conjecture of Krawchuk and Rampersad as an
immediate corollary:
\begin{corollary}
$\limsup_{n \rightarrow \infty} c(n)/n = 2$.
\end{corollary}

\section{Lower bound}
\label{four}

In this section we prove the the corresponding lower bound on $c(n)$,
namely
\begin{theorem}
$c(n) \geq {4\over 3} n - 4$ for $n \geq 0$.
\label{five}
\end{theorem}

The ideas are similar to those in the proof of the upper bound, but
a bit more complicated because the various exceptional sets 
are more intricate.

We need a lemma.
Define the following exceptional sets. 
\begin{align*}
A &= \{ 2^k - 1 \suchthat k \geq 1 \} = \{1,3,7,15,31, \ldots \} \\
B &= \{ 2^k + 1 \suchthat k \geq 2 \} = \{5,9,17,33, \ldots \}\\
D &= \{ 12 \cdot 2^k - 3 \suchthat k \geq 0 \} = \{9,21,45,93,\ldots \} \\
J &= \{ (2^{2i+1} +1)2^j \suchthat i \geq1, \ j \geq 0 \} =
	\{ 9, 18, 33, 36, 66, 72, 129, 132, 144, 258, \ldots \} .
\end{align*}

\begin{lemma}
\leavevmode
\begin{itemize}
\item[(i)]  If $n \in A$, then $c(n) \geq {4\over 3}n - 2$.
\item[(ii)] If $n \in B$, then $c(n) \geq {4\over 3}n - 4$.
\item[(iii)] If $n \in D$, then 
$c(n) = 2c((n+3)/4) + c((n+1)/2)$.
\item[(iv)] If $n \in J$, say $n = (2^{2i+1} + 1)2^j$, then
$c( n ) =  {8 \over 3} 2^{2i+j} + {4 \over 3} 2^j - 4 =
{4 \over 3} n - 4$.
for $i \geq 1$, $j \geq 0$.
\item[(v)] If $n \in 4J+3$, say $n = (2^{2i+1} +1) 2^{j+2} + 3$, and
$(i,j) \not= (1,0)$ (i.e., $n \not=39$), then
$c(n) = (104 \cdot 2^{2i+j} + 64 \cdot 2^j + 4\cdot 2^{2i} (-1)^j 
-10 \cdot (-1)^j + 18)/9 \geq (4n+16)/3$.
\item[(vi)] If $n \in 2J+3$, say $n = (2^{2i+1} +1) 2^{j+1} + 3$, and
furthermore $j \geq 1$, then
$c(n) = (52\cdot 2^{2i+j} + 32\cdot 2^j -4 \cdot 2^{2i} (-1)^j 
+ 10\cdot (-1)^j + 18)/9$.
\item[(vii)] If $n \in 2J-5$, say $n = (2^{2i+1} +1) 2^{j+1} -5$, 
then $c(n) = 6\cdot 2^{2i+j} +4 \cdot 2^j + {8\over 3} \cdot 2^{2i} (-1)^j
- {2\over 3}(-1)^j - 14$ for $i\geq 1$ and $j \geq 2$.
\end{itemize}
\label{six}
\end{lemma}

\begin{proof}
Items (i) and (ii) follow immediately from
Eqs.~\eqref{2k1} and \eqref{2km1}.

For item (iii), take $i = (n-1)/4$ in Eq.~\eqref{eq1}.   Then
$c(n) = 2c((n+3)/4) + c((n+1)/2)$ iff $a_1 ((n-1)/4) = 0$.
But from the DFAO in Figure~\ref{aut1} we see $a_i(m) = 0$ iff
$m = 3\cdot2^k - 1$ for $k \geq 0$.  
Hence $a_1 ((n-1)/4) = 0$ iff $(n-1)/4 = 3\cdot2^k - 1$ for $k \geq 0$,
iff $n = 12 \cdot 2^k - 3$ for $k \geq 0$.

For item (iv),
we use a variant of the linear representation trick.  Let
$n = (2^{2i+1} + 1)2^j $.  The base-$2$
representation of $n$
is $1 0^{2i} 1 0^j$, so $c(n) = v \gamma(1 0^{2i} 1 0^j) w
= v \gamma(1) \gamma(0)^{2i} \gamma(1) \gamma(0)^j w$.
The minimal polynomial of $\gamma(0)$ is $X^2 (X-1)(X-2)(X+1)$,
so each entry of $\gamma(0)^j$ for $j \geq 2$ is a linear combination of
$2^j, (-1)^j, $ and $1$.   The same is then true for
$\gamma(1) \gamma(0)^j w$.   Similarly, each entry of
$\gamma(0)^{2i}$ for $i \geq 2$ is a linear combination of
$2^{2i}$ and $1$, and the same is true for
$v \gamma(1) \gamma(0)^{2i}$.   Hence the entries of the product
$v \gamma(1) \gamma(0)^{2i} \gamma(1) \gamma(0)^j w$
are linear combinations of $2^{2i+j}$, $2^j$, $2^{2i}$, $1$,
$2^{2i} (-1)^j$, and $(-1)^j$.   We can deduce the particular
constants by substituting small values of $i$ and $j$ and
solving the resulting linear system.  The result now follows.

Parts (v), (vi), and (vii) follow from the same technique.
\end{proof}

We are now ready to prove Theorem~\ref{five}.

\begin{proof}
Again, the statement of the theorem does not seem strong enough
to make the induction go through.  

We will prove the following three claims below by 
simultaneous induction on $n$.
\begin{itemize}
\item[(i)] For all $n$ we have $c(n) \geq {4\over 3}n - 4$.
\item[(ii)] If $n$ is even and $n \not\in J$ 
then $c(n) \geq {4\over 3}n - 2$.
\item[(iii)] If $n$ is odd, $n \geq 47$, and $n \not\in A \uni B$,
then $c(n) \geq (4n+16)/3$.
\end{itemize}

The base case is $n < 191$, which we can check by a short computation.
Now assume $n \geq 191$.

\bigskip
\noindent{\it Case 1:}   $n \equiv \modd{0} {2}$, $n = 2i$.

\smallskip

Case 1a:  Suppose $n \in J$.   Then from Lemma~\ref{six} (iv) we have
$c(n) = {4\over 3}n - 4$.

\smallskip

Case 1b:  Suppose $n\not\in J$.   Then 
\begin{align*}
c(n) &= c(2i) = c(i) + a_0(i) \\
&= 2c(n/2) + a_0 (n/2) \\
&\geq 2 c(n/2) + 2 \\
&\geq 2 \cdot ({4\over 3} (n/2) - 2) + 2  \quad\quad \text{(by induction,
since $n/2 \geq 47$)} \\
&= {4 \over 3} n - 2.
\end{align*}

\noindent{\it Case 2:}   $n \equiv \modd{1} {4}$, $n = 4i+1$.
If $n \in A \uni B$ then the inequality $c(n) \geq {4\over 3}n - 4$
follows from Eqs.~\eqref{2k1} and \eqref{2km1}.  So assume
$n \not\in A \uni B$.

\smallskip

Case 2a:  Suppose $n \not\in D$,
${{n+3} \over 4} \not\in J \uni A \uni B$,
${{n+1}\over 2} \not\in A \uni B$.

\smallskip

Then 
\begin{align*}
c(n) &= c(4i+1) = 2c(i+1) + c(2i+1) + a_1(i)  \\
&= 2c((n+3)/4) + c((n+1)/2) + a_1 ((n-1)/4) \\
&\geq 2c((n+3)/4) + c((n+1)/2) + 2 \quad\quad \text{(because $n \not\in D$)} \\
&\geq 2 \cdot ({4\over 3} \cdot ((n+3)/4) - 2) + (4 {{n+1}\over 2} + 16)/3 + 2 \\
& \quad\quad \text{(because $(n+3)/4 \not\in J \uni A \uni B$ and
$(n+1)/2 \not\in A \uni B$} \\
& \quad\quad \text{and $(n+1)/2 \geq (n+3)/4 \geq 47$ by induction)} \\
& \geq {4\over 3} n + 6 \geq (4n+16)/3.
\end{align*}

Case 2b:  If $n \in D$ then Eq.~\eqref{12km3} implies that
$c(n) \geq (4n/3) - 4$.   If further $n \geq 47$ then it implies
that $c(n) \geq (4n+16)/3$.

\smallskip

Case 2c:  Suppose ${{n+3}\over 4} \in J$.   Then from Lemma~\ref{six} (v) we
have a closed form for $c(n)$.  Using
a routine calculation and the fact that $n \not= 39$, we get
$c(n) \geq (4n+16)/3$.

\smallskip

Case 2d:  Suppose ${{n+3}\over 4} \in A$.  
Then $n = 2^k - 7$ for $k \geq 3$, and then by Eq.~\eqref{2km7}
we have $c(n) \geq (4n+16)/3$ for $k \geq 5$.

\smallskip

Case 2e:  Suppose ${{n+3}\over 4} \in B$.    Then $n \in B$, a
contradiction.

\smallskip

Case 2f:  Suppose ${{n+1} \over 2} \in A$ .
Then $n = 2^k - 3$ for $k \geq 2$, and by Eq.~\eqref{2km3}
we have $c(n) \geq (4n+16)/3$ for $k \geq 5$.

\smallskip

Case 2g:  Suppose ${{n+1} \over 2} \in B$.  Then  $n \in B$,
a contradiction.

\bigskip

\noindent{\it Case 3:}  $n \equiv \modd{3} {4}$, $n = 4i+3$.
If $n \in A \uni B$ then $c(n) \geq {4\over 3}n - 4$
follows from Eqs.~\eqref{2k1} and \eqref{2km1}.  So assume $n \not\in A
\uni B$.

\smallskip

Case 3a:  Suppose $(n-3)/2 \not\in J \uni A \uni B$ and
$(n+3)/2 \not\in J \uni A \uni B$ and
$(n+5)/2 \not\in J \uni A \uni B$.   Then
\begin{align*} 
c(n) &= c(4i+3) = {1 \over 2} c(2i) + c(2i+3) + {1\over 2} c(2i+4) + a_3(i) \\
&= {1\over 2} c((n-3)/2) + c((n+3)/2) + c((n+5)/2) + a_3((n-3)/4) \\
&\geq {1\over 2} ({4\over 3} (n-3)/2 - 2) 
+ (4 {{n+3} \over 2} + 16)/3 + {1\over 2} ({4\over 3} (n+5)/2 - 2) - 1 \\
& \quad\quad \text{(by conditions on $(n-3)/2, (n+3)/2, (n+5)/2$} \\
& \quad\quad \text{and $(n+5)/2\geq (n+3)/2 \geq (n-3)/2 \geq 47)$} \\
&= {4\over 3}n + 6 \geq (4n+16)/3.
\end{align*}

Case 3b:  Suppose $(n-3)/2 \in J$.   Then $n \in 2J+3$, so
$n = (2^{2i+1} + 1)2^{j+1} + 3$.  Since $n \equiv \modd{3} {4}$ we
must have $j \geq 1$.   Then by a routine
calculation using Lemma~\ref{six} (vi),  we have $c(n) \geq (4n+16)/3$
since $n > 39$.

\smallskip

Case 3c:  Suppose $(n-3)/2 \in A \uni B$.
But $(n-3)/2 = 2n$ is even, a contradiction.

\smallskip

Case 3d:  Suppose $(n+3)/2 \in J$.  Then $n \in 2J - 3$.
But $n \equiv \modd{3} {4}$, so it is easy to see that
this forces $n = 2^{2k} - 1 \in B$ for $k \geq 2$, a contradiction.

\smallskip

Case 3e:  Suppose $(n+3)/2 \in A$.  Then $n = 2^k - 5$ and
by Eq.~\eqref{2km5} we have $c(n) \geq (4n+16)/3$ for $k \geq 6$.

\smallskip

Case 3f:  Suppose $(n+3)/2 \in B$.  Then $n \in A$, a contradiction.

\smallskip

Case 3g:  Suppose $(n+5)/2 \in J$.  Then $n \in 2J-5$.  
Then from Lemma~\ref{six} (vii) it follows by a routine computation that
$c(n) \geq (4n+16)/3$.

\smallskip

Case 3h:  Suppose $(n+5)/2 \in A \uni B$.   But
$(n+5)/2 = 2n+4$ is even, a contradiction.

\smallskip

This completes the proof by induction of (i), (ii), and (iii).
\end{proof}

\begin{corollary}
$\liminf_{n \rightarrow \infty} c(n)/n = 4/3$.
\end{corollary}

We also have enough to prove
\begin{theorem}
$c(n) = {4 \over 3} n - 4$ iff $n \in J$.
\end{theorem}

\begin{proof}
$\Longleftarrow$: follows from Lemma~\ref{six} (v).

$\Longrightarrow$:  if $n$ is even and $n \not\in J$ then
$c(n) \geq {4\over 3}n - 2 > {4\over 3}n - 4$ by above.   

If $n\geq 47$ is odd
and $n \not\in A \uni B$, then $c(n) \geq (4n+16)/3 > {4\over 3} n - 4$
by above.   It remains to check $n < 47$ and $n \in A \uni B$.

For $n \in A$, we know from Lemma~\ref{six} (i) that
$c(n) \geq {4 \over 3}n - 2 > {4\over 3} n - 4$.
For $n \in B$, it follows from Eq.~\eqref{2k1} that
$c(n) = {4 \over 3} n - 4$ iff $n = 2^{2k+1} + 1 \in J$.
Finally, $n < 47$ can be checked with a computation.
\end{proof}

\section{Final remarks}

For ordinary subword complexity $\rho_{\bf x}(n)$, which counts the number of
distinct length-$n$ factors appearing in $\bf x$,
it is known that if $\bf x$ is
an automatic sequence, then there is an automaton that
takes as input the representations of $n$ and $y$ in parallel,
and accepts iff $y = \rho_{\bf x}(n)$.   In other words,
$\rho_{\bf x}(n)$ is {\it synchronized\/} for automatic sequences;
see \cite{Goc&Schaeffer&Shallit:2013} for more details.
This means that checking 
whether $\rho_{\bf x}(n) \leq An+B$ or $\rho_{\bf x}(n) \geq
An+B$ for all $n$ is, in general,
decidable for automatic sequences,
since we can express these assertions as a first-order
logical formula.

However, for cyclic complexity the function $c_{\bf x}(n)$ is not, in general,
synchronized.  We can see this as follows:   let ${\bf p} = 11010001\cdots$
be the characteristic sequence of the powers of $2$.  Then it is not
hard to see that $c_{\bf p}(n) = O(\log n)$ and 
$c_{\bf p}(2^n) = n+2$ for $n \geq 0$.
However, by a theorem about synchronized sequences \cite{Shallit:2021h},
this kind of growth rate is impossible.

This fundamental difference may help explain why it is so much harder to prove
inequalities for cyclic complexity.

\end{document}